\documentclass[11pt]{amsart}

\usepackage[T1]{fontenc}
\usepackage[utf8]{inputenc}
\usepackage[a4paper,margin=1in]{geometry}
\usepackage{amsmath,amssymb,amsthm}
\usepackage{mathtools}
\usepackage{bm}
\usepackage{enumitem}
\usepackage{microtype}
\usepackage[hidelinks]{hyperref}
\usepackage[nameinlink,noabbrev]{cleveref}

\numberwithin{equation}{section}
\theoremstyle{plain}
\newtheorem{theorem}{Theorem}[section]
\newtheorem{lemma}[theorem]{Lemma}
\newtheorem{corollary}[theorem]{Corollary}

\theoremstyle{definition}

\theoremstyle{remark}
\newtheorem{remark}[theorem]{Remark}

\AddToHook{env/theorem/begin}{\crefalias{section}{theorem}}
\AddToHook{env/lemma/begin}{\crefalias{section}{lemma}}
\AddToHook{env/corollary/begin}{\crefalias{section}{corollary}}
\AddToHook{env/proposition/begin}{\crefalias{section}{proposition}}
\AddToHook{env/definition/begin}{\crefalias{section}{definition}}
\AddToHook{env/remark/begin}{\crefalias{section}{remark}}

\crefname{theorem}{theorem}{theorems}
\Crefname{theorem}{Theorem}{Theorems}
\crefname{lemma}{lemma}{lemmas}
\Crefname{lemma}{Lemma}{Lemmas}
\crefname{corollary}{corollary}{corollaries}
\Crefname{corollary}{Corollary}{Corollaries}
\crefname{proposition}{proposition}{propositions}
\Crefname{proposition}{Proposition}{Propositions}
\crefname{remark}{remark}{remarks}
\Crefname{remark}{Remark}{Remarks}
\crefname{definition}{definition}{definitions}
\Crefname{definition}{Definition}{Definitions}

\newcommand{\RR}{\mathbb{R}}
\newcommand{\Ha}{H_a}
\newcommand{\dd}{\,\mathrm{d}}
\newcommand{\one}{\mathbf{1}}

\title[Two-Weight Fractional Integrals for Inverse-Square Operators]
{Sharp and Endpoint Two-Weight Fractional Integral Estimates for Schr\"odinger Operators with Inverse-Square Potentials}

\author{Haochen Liu}
\address{School of Mathematics and Physics, Qingdao University of Science and Technology}
\email{liuhc@mails.qust.edu.cn}

\author{Qinghao Yu}
\address{School of Mathematics and Physics, Qingdao University of Science and Technology}
\email{19707725295@163.com}

\author{Hongyan Zhou\texorpdfstring{\textsuperscript{*}}{*}}
\address{School of Mathematics and Physics, Qingdao University of Science and Technology}
\email{zhouhongyan@qust.edu.cn}
\thanks{Corresponding author: Hongyan Zhou, \href{mailto:zhouhongyan@qust.edu.cn}{zhouhongyan@qust.edu.cn}.}

\date{}

\begin{document}

\begin{abstract}
Let $\Ha=-\Delta+a|x|^{-2}$ be the Friedrichs extension on $L^2(\RR^d)$, where $d\ge 3$ and $-(d-2)^2/4\le a<0$ is in the attractive Hardy range. Starting from the known positive two-sided comparison for the kernel of $\Ha^{-s/2}$, we determine the complete strong non-endpoint mapping range for two power weights. If
\[
 \sigma=\frac{d-2-\sqrt{(d-2)^2+4a}}{2},
 \qquad 0<s<d-2\sigma,
\]
then
\[
 \bigl\||x|^{-\beta}\Ha^{-s/2}f\bigr\|_{L^q}
 \lesssim
 \bigl\||x|^\alpha f\bigr\|_{L^p}
\]
holds for $1<p,q<\infty$ precisely under the exponent ordering, scaling, sum, and origin conditions stated in the main theorem. At either origin-critical boundary, the strong estimate and the corresponding weighted $L^p\to L^{q,\infty}$ estimate fail, whereas the Lorentz replacement $L^{p,1}\to L^{q,\infty}$ holds. We also derive weighted Sobolev consequences and treat the Hardy-critical Friedrichs case separately. No new heat-kernel, spectral multiplier, Bernstein, or Littlewood--Paley theorem is claimed.
\end{abstract}

\maketitle

\section{Introduction}
\label{sec:introduction}

Consider the Friedrichs extension
\begin{equation}
 \Ha=-\Delta+\frac{a}{|x|^2}
 \quad\text{on }L^2(\RR^d),
 \qquad d\ge3,
 \qquad -\frac{(d-2)^2}{4}\le a<0.
 \label{eq:operator}
\end{equation}
Set
\begin{equation}
 \sigma=\sigma(a)
 :=\frac{d-2-\sqrt{(d-2)^2+4a}}{2},
 \qquad 0<\sigma\le\frac{d-2}{2}.
 \label{eq:sigma}
\end{equation}
The attractive inverse-square potential is scale invariant and lies at the level of Hardy's inequality. Its heat kernel has singular factors at $x=0$ and $y=0$, rather than a plain Gaussian bound. Killip, Miao, Visan, Zhang, and Zheng proved the following positive two-sided fractional-kernel comparison \cite[Lemma~2.2]{KillipEtAl2018}: whenever
\begin{equation}
 0<s<d-2\sigma,
 \label{eq:s-range}
\end{equation}
one has, for $x,y\ne0$ and $x\ne y$,
\begin{equation}
 \Ha^{-s/2}(x,y)
 \simeq
 |x-y|^{s-d}
 \left(
 \frac{|x|}{|x-y|}\wedge
 \frac{|y|}{|x-y|}\wedge1
 \right)^{-\sigma}.
 \label{eq:kmmvzz-kernel}
\end{equation}

The classical Stein--Weiss theorem \cite{SteinWeiss1958} characterizes two-power-weight bounds for a single Riesz kernel. The right-hand side of \cref{eq:kmmvzz-kernel} contains three mechanisms: a near-diagonal Riesz singularity, an output-origin singularity, and an input-origin singularity. The purpose of this note is to determine their combined sharp mapping range.

Our main strong result states that
\begin{equation}
 \bigl\||x|^{-\beta}\Ha^{-s/2}f\bigr\|_{L^q(\RR^d)}
 \le C
 \bigl\||x|^\alpha f\bigr\|_{L^p(\RR^d)}
 \label{eq:main-estimate}
\end{equation}
holds for $1<p,q<\infty$ if and only if
\begin{align}
 p&\le q,
 \label{eq:condition-pq}\\
 \frac1q&=\frac1p+\frac{\alpha+\beta-s}{d},
 \label{eq:condition-scaling}\\
 \alpha+\beta&\ge0,
 \label{eq:condition-sum}\\
 \alpha+\sigma&<\frac d{p'},
 \qquad
 \beta+\sigma<\frac d q.
 \label{eq:condition-origin}
\end{align}
A self-contained direct proof of \cref{eq:condition-pq} is included in \cref{app:pq}.

At either equality in \cref{eq:condition-origin}, the corresponding strong estimate fails. A localized argument nevertheless gives the Lorentz replacement
\begin{equation}
 \bigl\||x|^{-\beta}\Ha^{-s/2}f\bigr\|_{L^{q,\infty}}
 \lesssim
 \bigl\||x|^\alpha f\bigr\|_{L^{p,1}}.
 \label{eq:endpoint-introduction}
\end{equation}
The corresponding weighted $L^p\to L^{q,\infty}$ estimate fails at either boundary.

The closest sharp radial result is the theorem of Nowak and Stempak for non-modified Hankel potentials \cite[Theorem~2.10]{NowakStempak2017}. After the radial Liouville transform, their individual origin conditions agree with ours, but their sum condition is weaker. The full-space obstruction $\alpha+\beta\ge0$ arises from translated test functions and is invisible in the radial class. General two-weight testing theorems of Sawyer and Kairema \cite{Sawyer1988,Kairema2013} provide an abstract context but do not state the explicit range in \cref{eq:condition-pq,eq:condition-scaling,eq:condition-sum,eq:condition-origin}.

The paper is organized as follows. \Cref{sec:model-main} introduces the model kernel and states the main results. \Cref{sec:necessity} proves the necessary conditions. \Cref{sec:sufficiency} gives the three-region decomposition and the sufficiency proof. \Cref{sec:lorentz} treats the critical Lorentz endpoints. \Cref{sec:sobolev} records weighted Sobolev consequences. \Cref{sec:hardy-critical} isolates the Hardy-critical case, and \cref{sec:previous-work} discusses earlier work and limitations. \Cref{app:pq} contains the direct proof of \cref{eq:condition-pq}.

\section{Model Kernel and Main Results}
\label{sec:model-main}

Define
\begin{equation}
 K_{s,\sigma}(x,y)
 =|x-y|^{s-d}
 \left(
 \frac{|x|}{|x-y|}\wedge
 \frac{|y|}{|x-y|}\wedge1
 \right)^{-\sigma}
 \label{eq:model-kernel}
\end{equation}
and
\begin{equation}
 T_{s,\sigma}f(x)
 =\int_{\RR^d}K_{s,\sigma}(x,y)f(y)\,\dd y.
 \label{eq:model-operator}
\end{equation}

\begin{theorem}[Sharp strong range for the model kernel]
\label{thm:model-strong}
Let $d\ge3$, $0<\sigma\le(d-2)/2$, and $0<s<d-2\sigma$. For $1<p,q<\infty$, the estimate
\begin{equation}
 \bigl\||x|^{-\beta}T_{s,\sigma}f\bigr\|_{L^q(\RR^d)}
 \le C\bigl\||x|^\alpha f\bigr\|_{L^p(\RR^d)}
 \label{eq:model-strong}
\end{equation}
holds if and only if \cref{eq:condition-pq,eq:condition-scaling,eq:condition-sum,eq:condition-origin} hold.
\end{theorem}

\begin{theorem}[Sharp strong range for the inverse-square operator]
\label{thm:operator-strong}
Let $\Ha$ be the operator in \cref{eq:operator}, let $\sigma$ be defined by \cref{eq:sigma}, and assume \cref{eq:s-range}. For $1<p,q<\infty$, the estimate in \cref{eq:main-estimate} holds if and only if \cref{eq:condition-pq,eq:condition-scaling,eq:condition-sum,eq:condition-origin} hold.
\end{theorem}

The Lorentz spaces below are taken with respect to Lebesgue measure and use the convention
\begin{equation}
 \|f\|_{L^{r,1}}
 =\int_0^\infty t^{1/r}f^*(t)\,\frac{\dd t}{t},
 \qquad
 \|f\|_{L^{r,\infty}}
 =\sup_{t>0}t^{1/r}f^*(t).
 \label{eq:lorentz-convention}
\end{equation}

\begin{theorem}[Origin-critical Lorentz endpoints]
\label{thm:endpoints}
Assume $d\ge3$, $0<\sigma\le(d-2)/2$, $0<s<d-2\sigma$, $1<p\le q<\infty$, and suppose that \cref{eq:condition-scaling,eq:condition-sum} hold.

If either
\begin{equation}
 \beta+\sigma=\frac d q
 \label{eq:output-critical}
\end{equation}
or
\begin{equation}
 \alpha+\sigma=\frac d{p'},
 \label{eq:input-critical}
\end{equation}
then
\begin{equation}
 \bigl\||x|^{-\beta}T_{s,\sigma}f\bigr\|_{L^{q,\infty}}
 \lesssim
 \bigl\||x|^\alpha f\bigr\|_{L^{p,1}}.
 \label{eq:lorentz-model}
\end{equation}
The same estimate holds with $T_{s,\sigma}$ replaced by $\Ha^{-s/2}$.

At either boundary, the corresponding strong $L^q$ estimate is false, and \cref{eq:lorentz-model} is false with $L^{p,1}$ replaced by $L^p$.
\end{theorem}

\begin{remark}
\label{rem:critical-boundaries}
The equalities in \cref{eq:output-critical,eq:input-critical} cannot occur simultaneously because $s+2\sigma<d$.
\end{remark}

\section{Necessary Conditions}
\label{sec:necessity}

We prove necessity for the model kernel. The lower comparison in \cref{eq:kmmvzz-kernel} then transfers the same tests to $\Ha^{-s/2}$.

\subsection{Scaling}

The kernel is homogeneous of degree $s-d$. Applying \cref{eq:model-strong} to $f_\lambda(x)=f(\lambda x)$ forces \cref{eq:condition-scaling}.

\subsection{Input concentration at the origin}

Fix the annulus $A=\{x\in\RR^d:2<|x|<3\}$ and let
\begin{equation}
 f_\varepsilon(y)
 =|y|^{-\alpha-d/p+\varepsilon}\one_{\{|y|<1/2\}}.
 \label{eq:input-test}
\end{equation}
Then
\begin{equation}
 \bigl\||y|^\alpha f_\varepsilon\bigr\|_{L^p}
 \simeq\varepsilon^{-1/p}.
 \label{eq:input-test-norm}
\end{equation}
For $x\in A$, the kernel is comparable to $|y|^{-\sigma}$ on the support of $f_\varepsilon$. The output contains
\begin{equation}
 \int_0^{1/2}
 r^{d-1-\sigma-\alpha-d/p+\varepsilon}\,\dd r.
 \label{eq:input-origin-integral}
\end{equation}
This integral diverges if $\alpha+\sigma>d/p'$. At equality, it is comparable to $\varepsilon^{-1}$, which is too large relative to \cref{eq:input-test-norm}. Thus the first condition in \cref{eq:condition-origin} is necessary and strict.

\subsection{Output concentration at the origin}

Let $f\ge0$ be smooth and supported in $\{y\in\RR^d:1<|y|<2\}$. For $|x|<1/4$,
\begin{equation}
 T_{s,\sigma}f(x)\gtrsim |x|^{-\sigma}.
 \label{eq:output-origin-lower}
\end{equation}
Strong $L^q$ boundedness therefore requires $q(\beta+\sigma)<d$, which is the second strict condition in \cref{eq:condition-origin}.

\subsection{The sum condition}

Translate a fixed pair of nearby small balls to a center $Re_1$. On these balls, the kernel is bounded below by a positive constant, while the weights are comparable to $R^\alpha$ and $R^{-\beta}$. Letting $R\to\infty$ gives \cref{eq:condition-sum}.

\subsection{The exponent ordering condition}

The direct multi-block construction in \cref{app:pq} uses geometrically separated balls whose radii are fixed fractions of their distances from the origin. Scaling makes every normalized block contribute a uniform amount to the output, and summing $M$ disjoint blocks forces $M^{1/q}\lesssim M^{1/p}$. This proves \cref{eq:condition-pq} and completes the necessity argument.

\section{Kernel Decomposition and Sufficiency}
\label{sec:sufficiency}

Set $r=|x-y|$ and introduce the disjoint regions
\begin{align}
 \Omega_0
 &=\left\{(x,y):r\le\tfrac14\min\{|x|,|y|\}\right\},
 \label{eq:omega-zero}\\
 \Omega_x
 &=\left\{(x,y):r>\tfrac14\min\{|x|,|y|\},\ |x|\le|y|\right\},
 \label{eq:omega-x}\\
 \Omega_y
 &=\left\{(x,y):r>\tfrac14\min\{|x|,|y|\},\ |y|<|x|\right\}.
 \label{eq:omega-y}
\end{align}

\begin{lemma}[Three-region bounds]
\label{lem:three-regions}
On the regions in \cref{eq:omega-zero,eq:omega-x,eq:omega-y}, respectively,
\begin{align}
 K_{s,\sigma}(x,y)&=r^{s-d},
 &&(x,y)\in\Omega_0,
 \label{eq:region-zero}\\
 K_{s,\sigma}(x,y)&\lesssim |x|^{-\sigma}r^{s+\sigma-d},
 &&(x,y)\in\Omega_x,
 \label{eq:region-x}\\
 K_{s,\sigma}(x,y)&\lesssim |y|^{-\sigma}r^{s+\sigma-d},
 &&(x,y)\in\Omega_y.
 \label{eq:region-y}
\end{align}
Moreover,
\begin{equation}
 K_{s,\sigma}(x,y)
 \le r^{s-d}
 +|x|^{-\sigma}r^{s+\sigma-d}
 +|y|^{-\sigma}r^{s+\sigma-d}.
 \label{eq:global-domination}
\end{equation}
\end{lemma}

\begin{proof}
On $\Omega_0$, both ratios in \cref{eq:model-kernel} exceed one, which gives \cref{eq:region-zero}. On $\Omega_x$, one has $|x|/r<4$. Separating the cases $|x|\le r$ and $r<|x|<4r$ gives \cref{eq:region-x}. The proof of \cref{eq:region-y} is symmetric. Finally, for $A,B>0$,
\[
 (A\wedge B\wedge1)^{-\sigma}
 =\max\{A^{-\sigma},B^{-\sigma},1\}
 \le A^{-\sigma}+B^{-\sigma}+1,
\]
which proves \cref{eq:global-domination}.
\end{proof}

We use the following operator form of the Stein--Weiss theorem. If $0<\lambda<d$, $1<p\le q<\infty$, and
\begin{equation}
 \frac1q=\frac1p+\frac{A+B-\lambda}{d},
 \qquad
 A<\frac d{p'},
 \qquad
 B<\frac d q,
 \qquad
 A+B\ge0,
 \label{eq:stein-weiss-conditions}
\end{equation}
then
\begin{equation}
 \bigl\||x|^{-B}I_\lambda f\bigr\|_{L^q}
 \lesssim
 \bigl\||x|^A f\bigr\|_{L^p}.
 \label{eq:stein-weiss}
\end{equation}

\begin{proof}[Proof of sufficiency in \cref{thm:model-strong}]
Since $s+2\sigma<d$, both $s$ and $s+\sigma$ belong to $(0,d)$. By \cref{eq:global-domination},
\begin{equation}
 T_{s,\sigma}|f|
 \lesssim
 I_s|f|
 +|x|^{-\sigma}I_{s+\sigma}|f|
 +I_{s+\sigma}(|\cdot|^{-\sigma}|f|).
 \label{eq:three-stein-weiss-pieces}
\end{equation}
Apply \cref{eq:stein-weiss} with
\[
 (\lambda,A,B)=(s,\alpha,\beta),
 \qquad
 (\lambda,A,B)=(s+\sigma,\alpha,\beta+\sigma),
\]
and
\[
 (\lambda,A,B)=(s+\sigma,\alpha+\sigma,\beta).
\]
The scaling relation is unchanged, the individual upper restrictions follow from \cref{eq:condition-origin}, and the sum restrictions follow from \cref{eq:condition-sum}. Density in the weighted input space gives the bounded extension.
\end{proof}

\begin{proof}[Proof of \cref{thm:operator-strong}]
The upper comparison in \cref{eq:kmmvzz-kernel} transfers sufficiency from \cref{thm:model-strong}. The positive lower comparison transfers every necessity test. On a spectral core with compact spectral support away from zero, the kernel operator agrees with the spectral operator $\Ha^{-s/2}$. The weighted bounded extension is unique and therefore agrees with the spectral operator wherever both are defined.
\end{proof}

\section{Critical Lorentz Endpoints}
\label{sec:lorentz}

We first record the Lorentz pairing used below.

\begin{lemma}[Lorentz pairing]
\label{lem:lorentz-pairing}
For $1<r<\infty$,
\begin{equation}
 \int_{\RR^d}|u(x)v(x)|\,\dd x
 \lesssim
 \|u\|_{L^{r,1}}\|v\|_{L^{r',\infty}}.
 \label{eq:lorentz-pairing}
\end{equation}
Moreover,
\begin{equation}
 |x|^{-d/r}\in L^{r,\infty}(\RR^d)\setminus L^r(\RR^d).
 \label{eq:critical-power-lorentz}
\end{equation}
\end{lemma}

\begin{proof}
The Hardy--Littlewood rearrangement inequality gives
\[
 \int_{\RR^d}|uv|
 \le\int_0^\infty u^*(t)v^*(t)\,\dd t.
\]
Using $v^*(t)\le \|v\|_{L^{r',\infty}}t^{-1/r'}$ and the convention in \cref{eq:lorentz-convention} proves \cref{eq:lorentz-pairing}. The second assertion follows from
\[
 \bigl|\{x\in\RR^d:|x|^{-d/r}>\lambda\}\bigr|
 =\omega_d\lambda^{-r}.
\]
This is the classical Lorentz framework associated with O'Neil \cite{ONeil1963}; the particular pairing needed here has been proved directly.
\end{proof}

\subsection{The output-critical boundary}

Assume \cref{eq:output-critical}. The scaling relation gives
\begin{equation}
 \alpha=s+\sigma-\frac d p.
 \label{eq:alpha-output-critical}
\end{equation}
The near-diagonal and input-origin pieces satisfy strict Stein--Weiss inequalities and map the weighted $L^{p,1}$ space into strong $L^q$.

On $\Omega_x$, geometry gives $|x-y|\simeq |y|$. With $g(y)=|y|^\alpha|f(y)|$, one has
\begin{align}
 |x|^{-\beta}|T_xf(x)|
 &\lesssim |x|^{-\beta-\sigma}
 \int_{\RR^d}|y|^{s+\sigma-d}|f(y)|\,\dd y
 \notag\\
 &=|x|^{-d/q}
 \int_{\RR^d}|y|^{-d/p'}g(y)\,\dd y.
 \label{eq:output-critical-piece}
\end{align}
Applying \cref{lem:lorentz-pairing} proves the weak $L^q$ bound. The finite quasi-triangle inequality in $L^{q,\infty}$ then yields \cref{eq:lorentz-model}.

Strong $L^q$ boundedness fails for a fixed input supported away from the origin. To disprove the weighted $L^p\to L^{q,\infty}$ estimate, choose $1/p<\theta<1$ and set
\begin{equation}
 g(y)
 =|y|^{-d/p}\bigl[\log(e/|y|)\bigr]^{-\theta}
 \one_{\{|y|<1/2\}},
 \qquad
 f(y)=|y|^{-\alpha}g(y).
 \label{eq:output-critical-counterexample}
\end{equation}
Then $g\in L^p$, whereas for $0<|x|\ll1$,
\begin{equation}
 |x|^{-\beta}T_{s,\sigma}f(x)
 \gtrsim
 |x|^{-d/q}\bigl[\log(e/|x|)\bigr]^{1-\theta},
 \label{eq:output-critical-counterexample-lower}
\end{equation}
which does not belong to weak $L^q$.

\subsection{The input-critical boundary}

Assume \cref{eq:input-critical}. Scaling gives
\begin{equation}
 \beta=s+\sigma-\frac d{q'}.
 \label{eq:beta-input-critical}
\end{equation}
The near-diagonal and output-origin pieces are interior strong estimates. On $\Omega_y$, one has $|x-y|\simeq |x|$, and with $g(y)=|y|^\alpha|f(y)|$,
\begin{equation}
 |x|^{-\beta}|T_yf(x)|
 \lesssim
 |x|^{-d/q}
 \int_{\RR^d}|y|^{-d/p'}g(y)\,\dd y.
 \label{eq:input-critical-piece}
\end{equation}
The Lorentz pairing in \cref{lem:lorentz-pairing} proves \cref{eq:lorentz-model}. The input-origin truncated-power test rules out the weighted $L^p\to L^{q,\infty}$ estimate. A fixed nonnegative input in a small annulus around the origin produces the non-$L^q$ tail $|x|^{-d/q}$, so $L^{p,1}\to L^q$ also fails.

The two-sided comparison in \cref{eq:kmmvzz-kernel} transfers the endpoint estimates and counterexamples to $\Ha^{-s/2}$, completing the proof of \cref{thm:endpoints}.

\section{Weighted Sobolev Consequences}
\label{sec:sobolev}

Let $u=\psi(\Ha)v$, where $v\in L^2(\RR^d)$ and $\psi\in C_c^\infty((0,\infty))$. Then
\begin{equation}
 u=\Ha^{-s/2}\Ha^{s/2}u.
 \label{eq:spectral-identity}
\end{equation}
Applying \cref{thm:operator-strong} gives the following consequence.

\begin{corollary}[Two-weight Sobolev estimate]
\label{cor:two-weight-sobolev}
Under the hypotheses of \cref{thm:operator-strong},
\begin{equation}
 \bigl\||x|^{-\beta}u\bigr\|_{L^q}
 \lesssim
 \bigl\||x|^\alpha\Ha^{s/2}u\bigr\|_{L^p}
 \label{eq:weighted-sobolev}
\end{equation}
whenever \cref{eq:condition-pq,eq:condition-scaling,eq:condition-sum,eq:condition-origin} hold. The estimate extends to the completion of the spectral core in the norm on the right-hand side of \cref{eq:weighted-sobolev}.
\end{corollary}

The endpoint theorem similarly yields a Lorentz Sobolev estimate at either origin-critical boundary, with $L^{p,1}$ on the right and $L^{q,\infty}$ on the left.

\begin{corollary}[Same-exponent weighted Hardy--Sobolev family]
\label{cor:hardy-sobolev}
Let $1<p<\infty$ and suppose that
\begin{equation}
 s+\sigma-\frac d p
 <\alpha
 <\frac d{p'}-\sigma.
 \label{eq:alpha-range}
\end{equation}
Then
\begin{equation}
 \bigl\||x|^{\alpha-s}u\bigr\|_{L^p}
 \lesssim
 \bigl\||x|^\alpha\Ha^{s/2}u\bigr\|_{L^p}.
 \label{eq:hardy-sobolev-family}
\end{equation}
\end{corollary}

\begin{proof}
Set $q=p$ and $\beta=s-\alpha$. The scaling and sum conditions are automatic, and the two origin restrictions are exactly \cref{eq:alpha-range}.
\end{proof}

The interval in \cref{eq:alpha-range} has length $d-s-2\sigma>0$. Weighted comparison estimates between powers of $\Ha$ and $-\Delta$ were proved by Bui, D'Ancona, Duong, Li, and Ly \cite[Theorem~1.1]{BuiEtAl2017}. \Cref{cor:hardy-sobolev} also applies at parameter values outside their comparison range, so it is not merely obtained by replacing $\Ha^{s/2}$ with $(-\Delta)^{s/2}$ and invoking the classical Stein--Weiss theorem.

Classical Caffarelli--Kohn--Nirenberg inequalities \cite{CaffarelliKohnNirenberg1984}, together with quadratic-form comparison, already imply several strictly subcritical $p=2$ energy estimates. Those estimates are not presented here as new contributions.

\section{The Hardy-Critical Friedrichs Case}
\label{sec:hardy-critical}

At
\begin{equation}
 a=-\frac{(d-2)^2}{4},
 \qquad
 \sigma=\frac{d-2}{2},
 \label{eq:hardy-critical-parameters}
\end{equation}
the kernel condition in \cref{eq:s-range} becomes
\begin{equation}
 0<s<2.
 \label{eq:hardy-critical-s-range}
\end{equation}
The two-sided kernel comparison of Killip et al. remains valid for the Friedrichs extension, without an additional logarithmic factor in this range. The critical positive harmonic profile is not an $L^2$ zero eigenfunction; the singular zero-energy behavior is already encoded by the factors in \cref{eq:kmmvzz-kernel}.

The strong and Lorentz theorems therefore remain valid under \cref{eq:hardy-critical-s-range}, subject to their respective weight conditions. However, familiar unweighted energy embeddings may lie exactly on an excluded origin boundary. For example, with $p=2$, $s=1$, $\alpha=\beta=0$, and $q=2d/(d-2)$,
\begin{equation}
 \beta+\sigma
 =\frac{d-2}{2}
 =\frac d q.
 \label{eq:critical-energy-boundary}
\end{equation}
Thus the strong estimate
\begin{equation}
 \|u\|_{L^{2d/(d-2)}}
 \lesssim
 \|\Ha^{1/2}u\|_{L^2}
 \label{eq:critical-energy-estimate}
\end{equation}
is not supplied at the critical coupling. This boundary behavior is one reason to state the Hardy-critical Friedrichs case separately.

\section{Relation to Previous Work and Limitations}
\label{sec:previous-work}

The fractional kernel in \cref{eq:kmmvzz-kernel}, together with multiplier, Littlewood--Paley, Hardy, and unweighted Sobolev results for $\Ha$, is due to Killip et al. \cite{KillipEtAl2018}. The contribution of the present note is the explicit power-weight mapping classification for that known kernel.

Nowak and Stempak \cite[Theorem~2.10]{NowakStempak2017} proved a sharp theorem for the one-dimensional non-modified Hankel potential. In the radial inverse-square sector, the Liouville transform uses the Hankel index
\begin{equation}
 \nu
 =\sqrt{\frac{(d-2)^2}{4}+a}
 =\frac{d-2}{2}-\sigma.
 \label{eq:hankel-index}
\end{equation}
Their individual power restrictions translate into those in \cref{eq:condition-origin}. Their radial sum condition becomes
\begin{equation}
 \alpha+\beta
 +(d-1)\left(\frac1p-\frac1q\right)
 \ge0,
 \label{eq:radial-sum-condition}
\end{equation}
which is weaker than \cref{eq:condition-sum}. The additional full-space restriction is forced by translated balls and cannot be detected by radial test functions. Mode-by-mode Hankel estimates do not provide the uniform vector-valued reconstruction required for arbitrary full-space data.

Sawyer's classical theorem and Kairema's metric-space extension characterize general two-weight bounds through testing conditions \cite{Sawyer1988,Kairema2013}. These results do not directly state the five explicit power conditions obtained here. Moreover, the fixed origin singularities make the inclusion of the undecomposed kernel in the standard potential-type monotonicity class nontransparent. Our proof instead resolves the kernel into three classical Stein--Weiss pieces.

The explicit full-space strong range and the two origin-critical Lorentz replacements do not appear in the literature checked by the authors. This is a cautious literature-priority statement, not a claim of exhaustive bibliographic completeness. Recent work on endpoints of the classical Stein--Weiss inequality includes \cite{SunWang2026}, but the inverse-square max-type kernel requires the localized argument in \cref{sec:lorentz}.

We do not address $p=1$, $q=\infty$, best constants, extremizers, other self-adjoint extensions, or general Schr\"odinger potentials. We claim no new heat-kernel, spectral multiplier, Bernstein, or Littlewood--Paley theory, and no solution of the broader open program for harmonic analysis of arbitrary Schr\"odinger operators.

\appendix
\section{A Direct Proof of the Exponent Ordering}
\label{app:pq}

Assume \cref{eq:model-strong} and the necessary scaling relation \cref{eq:condition-scaling}. Fix $c\in(0,10^{-2})$ and set
\begin{equation}
 R_j=10^j,
 \qquad
 x_j=R_je_1,
 \qquad
 r_j=cR_j,
 \label{eq:appendix-scales}
\end{equation}
and
\begin{equation}
 B_j=B(x_j,r_j),
 \qquad
 E_j=B(x_j,r_j/4).
 \label{eq:appendix-balls}
\end{equation}
These families are pairwise disjoint, and $|x|\simeq R_j$ on each corresponding ball. Define
\begin{equation}
 A_j=R_j^{-\alpha}r_j^{-d/p},
 \qquad
 f_j=A_j\one_{B_j},
 \qquad
 f_M=\sum_{j=1}^M f_j.
 \label{eq:appendix-functions}
\end{equation}
Then
\begin{equation}
 \bigl\||x|^\alpha f_M\bigr\|_{L^p}^p
 \simeq M.
 \label{eq:appendix-input-norm}
\end{equation}
For $x\in E_j$, the ball $B(x,r_j/8)$ is contained in $B_j$, and the kernel is the classical near-diagonal kernel there. Positivity gives
\begin{equation}
 T_{s,\sigma}f_M(x)
 \ge A_j\int_{|z|<r_j/8}|z|^{s-d}\,\dd z
 \gtrsim A_jr_j^s.
 \label{eq:appendix-output-lower}
\end{equation}
The weighted $L^q$ norm on $E_j$ is therefore bounded below by
\begin{equation}
 R_j^{-\beta}A_jr_j^s|E_j|^{1/q}
 \simeq
 R_j^{-\alpha-\beta}r_j^{s-d/p+d/q}.
 \label{eq:appendix-block-lower}
\end{equation}
By \cref{eq:condition-scaling},
\begin{equation}
 s-\frac dp+\frac dq=\alpha+\beta.
 \label{eq:appendix-scaling-identity}
\end{equation}
Because $r_j=cR_j$, the expression in \cref{eq:appendix-block-lower} is bounded below by a positive constant independent of $j$. Since the sets $E_j$ are disjoint,
\begin{equation}
 \bigl\||x|^{-\beta}T_{s,\sigma}f_M\bigr\|_{L^q}
 \gtrsim M^{1/q}.
 \label{eq:appendix-output-norm}
\end{equation}
Combining \cref{eq:model-strong,eq:appendix-input-norm,eq:appendix-output-norm} gives
\[
 M^{1/q}\lesssim M^{1/p}
\]
for every positive integer $M$. Hence $p\le q$.

\bibliographystyle{amsplain}
\bibliography{references}

\providecommand{\bysame}{\leavevmode\hbox to3em{\hrulefill}\thinspace}
\providecommand{\MR}{\relax\ifhmode\unskip\space\fi MR }
\providecommand{\MRhref}[2]{%
  \href{http://www.ams.org/mathscinet-getitem?mr=#1}{#2}
}
\providecommand{\href}[2]{#2}
\begin{thebibliography}{1}

\bibitem{BuiEtAl2017}
The~Anh Bui, Piero D'Ancona, Xuan~Thinh Duong, Ji~Li, and Fu~Ken Ly,
  \emph{Weighted estimates for powers and smoothing estimates of
  {Schr\"odinger} operators with inverse-square potentials}, Journal of
  Differential Equations \textbf{262} (2017), no.~3, 2771--2807.

\bibitem{CaffarelliKohnNirenberg1984}
Luis Caffarelli, Robert Kohn, and Louis Nirenberg, \emph{First order
  interpolation inequalities with weights}, Compositio Mathematica \textbf{53}
  (1984), no.~3, 259--275.

\bibitem{Kairema2013}
Anna Kairema, \emph{Two-weight norm inequalities for potential type and maximal
  operators in a metric space}, Publicacions Matem\`atiques \textbf{57} (2013),
  no.~1, 3--56.

\bibitem{KillipEtAl2018}
Rowan Killip, Changxing Miao, Monica Visan, Jiqiang Zhang, and Jialing Zheng,
  \emph{Sobolev spaces adapted to the {Schr\"odinger} operator with
  inverse-square potential}, Mathematische Zeitschrift \textbf{288} (2018),
  no.~3--4, 1273--1298.

\bibitem{NowakStempak2017}
Adam Nowak and Krzysztof Stempak, \emph{Potential operators associated with
  {Hankel} and {Hankel--Dunkl} transforms}, Journal d'Analyse Math\'ematique
  \textbf{131} (2017), no.~1, 277--321.

\bibitem{ONeil1963}
Richard O'Neil, \emph{Convolution operators and {$L(p,q)$} spaces}, Duke
  Mathematical Journal \textbf{30} (1963), no.~1, 129--142.

\bibitem{Sawyer1988}
Eric~T. Sawyer, \emph{A characterization of two weight norm inequalities for
  fractional and {Poisson} integrals}, Transactions of the American
  Mathematical Society \textbf{308} (1988), no.~2, 533--545.

\bibitem{SteinWeiss1958}
Elias~M. Stein and Guido Weiss, \emph{{Fractional Integrals on
  {$n$}-Dimensional {Euclidean} Space}}, Journal of Mathematics and Mechanics
  \textbf{7} (1958), no.~4, 503--514.

\bibitem{SunWang2026}
Chuhan Sun and Zipeng Wang, \emph{{On the End-Point of {Stein--Weiss}
  Inequality}}, The Journal of Geometric Analysis \textbf{36} (2026), no.~3,
  100.

\end{thebibliography}

\end{document}